\documentclass[11pt]{amsart}

\usepackage[a4paper,hmargin=3.5cm,vmargin=4cm]{geometry}
\usepackage{amsfonts,amssymb,amscd,amstext,verbatim}

\usepackage{fancyhdr}
\usepackage[utf8]{inputenc}
\usepackage{hyperref}
\usepackage{verbatim}
\input xy
\xyoption{all}

\usepackage{times}

\usepackage{enumerate}
\usepackage{titlesec}
\usepackage{mathrsfs}

\titleformat{\section}
{\filcenter\bfseries\large} {\thesection{.}}{0.2cm}{}
\titleformat{\subsection}[runin]
{\bfseries} {\thesubsection{.}}{0.15cm}{}[.]
\titleformat{\subsubsection}[runin]
{\em}{\thesubsubsection{.}}{0.15cm}{}[.]

\usepackage[up,bf]{caption}

\newtheorem{theorem}{Theorem}[section]
\newtheorem{proposition}[theorem]{Proposition}

\newtheorem{lemma}[theorem]{Lemma}
\newtheorem{corollary}[theorem]{Corollary}

\theoremstyle{definition}
\newtheorem{definition}[theorem]{Definition}
\newtheorem{remark}[theorem]{Remark}

\newtheorem{example}[theorem]{Example}

\numberwithin{equation}{section}
\numberwithin{figure}{section}

\newcommand\Scal{\mathcal{S}}

\newcommand\Cscr{\mathscr{C}}

\newcommand\Jscr{\mathscr{J}}

\newcommand\Oscr{\mathscr{O}}

\newcommand\B{\mathbb{B}}
\newcommand\C{\mathbb{C}}
\newcommand\D{\overline{\mathbb D}}
\newcommand\CP{\mathbb{CP}}
\newcommand\HP{\mathbb{HP}}
\renewcommand\D{\mathbb D}

\renewcommand\H{\mathbb{H}}

\renewcommand\P{\mathbb{P}}
\newcommand\R{\mathbb{R}}

\newcommand\Z{\mathbb{Z}}

\newcommand\igot{\mathfrak{i}}
\newcommand\jgot{\mathfrak{j}}
\newcommand\kgot{\mathfrak{k}}
\renewcommand\igot{\mathfrak{i}}

\newcommand\ngot{\mathfrak{n}}

\newcommand\sgot{\mathfrak{s}}

\newcommand\E{\mathrm{e}}

\renewcommand\imath{\igot}

\newcommand\hra{\hookrightarrow}
\newcommand\lra{\longrightarrow}

\newcommand\wt{\widetilde}

\newcommand\cd{\overline{\D}}

\newcommand\Id{\mathrm{Id}}

\newcommand\Sym{\mathrm{Sym}}
\newcommand\End{\mathrm{End}}
\newcommand\SM{\mathrm{SM}}

\numberwithin{equation}{section}

\begin{document}

\fancyhead[LO]{The Calabi-Yau property of superminimal surfaces}
\fancyhead[RE]{F.\ Forstneri{\v c}} 
\fancyhead[RO,LE]{\thepage}

\thispagestyle{empty}

\vspace*{1cm}
\begin{center}
{\bf\LARGE The Calabi-Yau property of superminimal surfaces in self-dual Einstein four-manifolds}

\vspace*{0.5cm}

{\large\bf  Franc Forstneri{\v c}} 
\end{center}

\vspace*{1cm}

\begin{quote}
{\small
\noindent {\bf Abstract}\hspace*{0.1cm}
In this paper, we show that if $(X,g)$ is an oriented four dimensional Einstein manifold 
which is self-dual or anti-self-dual then superminimal surfaces in $X$ of appropriate spin 
enjoy the Calabi-Yau property, meaning that every immersed 
surface of this type from a bordered Riemann surface can be uniformly approximated by  complete 
superminimal surfaces with Jordan boundaries. The proof uses the theory of twistor spaces and the
Calabi-Yau property of holomorphic Legendrian curves in complex contact manifolds.

\vspace*{0.2cm}

\noindent{\bf Keywords}\hspace*{0.1cm}  superminimal surface, Einstein manifold, twistor space, 
complex contact manifold, holomorphic Legendrian curve 

\vspace*{0.1cm}

\noindent{\bf MSC (2020):}\hspace*{0.1cm}} 
Primary 53A10 - 53C25 - 53C28. Secondary 32E30 - 37J55.

\vspace*{0.1cm}
\noindent{\bf Date: \rm 29 June 2020}

%
%
%
%
\end{quote}

%
%
%
%
\section{Introduction}\label{sec:intro}
It has been known since the 1980s that four dimensional self-dual Einstein manifolds have a rich 
theory of {\em superminimal surfaces}. 
In the present paper we provide further evidence by showing that such surfaces 
enjoy the {\em Calabi-Yau property}; see Theorems \ref{th:main} and \ref{th:main2}.
The latter term was introduced in the recent paper by Alarc\'on et al.\ 
\cite[Definition 6.1]{AlarconForstnericLarusson2019X}. The motivation comes from the classical problem 
posed by Calabi in 1965 (see \cite[p.\ 170]{Calabi1965Conjecture} and \cite[p.\ 212]{Chern1966BAMS}) 
and in a more precise form by S.-T.\ Yau in 2000 (see \cite[p.\ 360]{Yau2000AMS} and 
\cite[p.\ 241]{Yau2000AJM}),
asking which open Riemann surfaces admit complete conformal minimal immersions with 
bounded images into Euclidean spaces $\R^n$, $n\ge 3$, and what is the possible boundary behaviour
of such surfaces. For the history of this subject and some recent developments, see the survey 
\cite{AlarconForstneric2019JAMS} and the papers \cite{AlarconDrinovecForstnericLopez2015PLMS,AlarconForstneric2020RMI,AlarconForstnericLopez2020MAMS}.

Superminimal surfaces form an interesting class of minimal surfaces in four dimensional Riemannian 
manifolds. Although this term was coined by Bryant in his study \cite{Bryant1982JDG} 
of such surfaces in the four-sphere $S^4$ and their relationship to holomorphic Legendrian 
curves in $\CP^3$, the Penrose twistor space of $S^4$, it soon became clear  through the 
work of Friedrich \cite{Friedrich1984,Friedrich1997} that this class of minimal surfaces was described 
geometrically already by Kommerell in his 1897 dissertation \cite{Kommerell1897} and his 1905 paper \cite{Kommerell1905}, 
and they were subsequently studied by Eisenhart \cite{Eisenhart1912} (1912), 
Bor$\mathring{\textrm u}$vka \cite{Boruvka1928ii,Boruvka1928i} (1928), Calabi \cite{Calabi1967},  
and Chern \cite{Chern1970A,Chern1970B} (1970), among others; see Sect.\ \ref{sec:SM}.
Unfortunately, at least three different definitions are used in the literature. We 
adopt the original geometric definition of Kommerell \cite{Kommerell1897} (see also 
Friedrich \cite[Sect. 1]{Friedrich1997}) and explain the role of {\em spin} in this context. 

Assume that $(X,g)$ is a Riemannian four-manifold and $M\subset X$ is 
a smoothly embedded surface with the induced conformal structure. 
(Our considerations, being of local nature, will also apply to immersed surfaces.) 
Then $TX|_M=TM\oplus N$ where $N=N(M)$ is the orthogonal normal 
bundle to $M$. A unit normal vector $n\in N_x$ at a point $x\in M$ determines a
{\em second fundamental form} $S_x(n):T_xM\to T_xM$, a self-adjoint linear operator
on the tangent space of $M$. For a fixed tangent vector $v\in T_x M$ we consider the closed curve 
\begin{equation}\label{eq:Ixv}
	I_x(v)=\bigl\{S_x(n)v: n\in N_x,\ |n|_g=1\bigr\} \subset T_xM.
\end{equation}
Suppose now that $M$ and $X$ are oriented, and coorient the normal bundle $N$ accordingly.

%
%
\begin{definition}\label{def:SM}
A smooth oriented embedded surface $M$ in an oriented Riemannian four-manifold
$(X,g)$ is {\em superminimal of positive (negative) spin} if for every point $x\in M$ and unit tangent vector 
$v\in T_xM$, the curve $I_x(v)\subset T_xM$ \eqref{eq:Ixv} is a circle centred at $0$ 
and the map $n \to S(n)v \in I_x(v)$ $(n\in N_x)$ 
is orientation preserving (resp.\ orientation reversing). The last condition 
is void at points $x\in M$ where the circle $I_x(v)$ reduces to $0\in T_xM$.
The analogous definition applies to a smoothly immersed oriented surface $f:M\to X$. 
\end{definition}

Every superminimal surface is a minimal surface; see Friedrich \cite[Proposition 3]{Friedrich1997}
and the discussion in Sect.\ \ref{sec:SM}. 
The converse is not true except in special cases, see Remark \ref{rem:biquadratic}.
The notion of spin, which is only implicitly present in Friedrich's discussion, is very important in 
the {\em Bryant correspondence} described in Theorem \ref{th:Friedrich}.

The surface $M$ in Definition \ref{def:SM} is endowed with the conformal structure which 
renders the given immersion $M\to X$ conformal. In the sequel we prefer to work with a fixed conformal 
structure on $M$ and consider only conformal immersions $M\to X$. Since $M$ is also oriented, 
it is a Riemann surface. We denote by $\SM^\pm(M,X)$ the spaces of smooth conformal superminimal
immersions of positive and negative spin, respectively, and set
\begin{equation}\label{eq:SM}
	\SM(M,X)=\SM^+(M,X) \cup \SM^-(M,X).
\end{equation}
The intersection $\SM^+(M,X) \cap \SM^-(M,X)$ of these two spaces consists of immersions
for which all circles $I_x(v)$ \eqref{eq:Ixv} reduce to points; such surfaces are minimal with 
identically vanishing normal curvature, hence totally geodesic (see \cite{Friedrich1997}).

Recall that a {\em (finite) bordered Riemann surface} is a domain of the form 
$M=R\setminus \bigcup_i \Delta_i$, where $R$ is a compact Riemann surface 
and $\Delta_i$ are finitely many compact pairwise disjoint discs with  
smooth boundaries $b\Delta_i$, diffeomorphic images of $\cd=\{z\in\C:|z|\le 1\}$.
Its closure $\overline M$ is a {\em compact bordered Riemann surface}.
The definition of superminimality clearly applies to smooth conformal
immersions $\overline M\to X$ and the notation \eqref{eq:SM} shall be used accordingly.
The following is our first main result; see also Theorem \ref{th:main2}.

%
%
\begin{theorem}
\label{th:main}
Let $(X,g)$ be an oriented four dimensional Einstein manifold whose Weyl tensor $W=W^++W^-$ 
satisfies $W^+=0$ or $W^-=0$. Given any bordered Riemann surface $M$ and a conformal
superminimal immersion $f_0\in \SM^\pm(\overline M,X)$ of class $\Cscr^3$ 
(with the respective choice of sign $\pm$), we can approximate $f_0$ uniformly on $\overline M$ 
by continuous maps $f:\overline M\to X$ such that $f:M\to X$ is a complete conformal superminimal immersion
in $\SM^\pm(M,X)$ and $f:bM\to X$ is a topological embedding.
\end{theorem}

Recall that an immersion $f:M\to (X,g)$ is said to be {\em complete} if the Riemannian metric $f^*g$ 
induced by the immersion is a complete metric on $M$; equivalently, for any divergent path $\lambda:[0,1)\to M$ 
(i.e., such that $\lambda(t)$ leaves any compact subset of $M$ as $t\to 1$) the path $f\circ \lambda:[0,1)\to X$
has infinite length:  $\int_0^1 \big|\frac{d (f\circ \lambda(t))}{dt}\big|_g dt=+\infty$.

Note that our result is local in the sense that the complete conformal superminimal immersion
stays uniformly close to the given superminimal surface. Hence, if Theorem \ref{th:main} holds for a 
Riemannian manifold $X$ then it also holds for every open domain in $X$.

Recall (see Atiyah et al. \cite[p.\ 427]{AtiyahHitchinSinger1978}) that the Weyl tensor $W=W^++W^-$
is the conformally invariant part of the curvature tensor of a Riemannian four-manifold $(X,g)$, so it only 
depends on the conformal class of the metric. The manifold is called {\em self-dual} if $W^-=0$,  and
{\em anti-self-dual} if $W^+=0$. Note that $W=0$ if and only if the metric is conformally flat.
A Riemannian manifold $(X,g)$ is called an {\em Einstein manifold} if the Ricci tensor of $g$ is proportional 
to the metric, $Ric_g=kg$ for some constant $k\in \R$. 
The curvature tensor of $g$ then reduces to the constant scalar curvature (the trace of the Ricci curvature,
hence $4k$ when $\dim X=4$) and the Weyl tensor $W$ (see \cite[p.\ 427]{AtiyahHitchinSinger1978}). 
The Einstein condition is equivalent to the metric being a solution of the vacuum Einstein field equations
with a cosmological constant,  although the signature of the metric can be arbitrary in this setting, 
thus not being restricted to the four-dimensional Lorentzian manifolds studied in general relativity. 
Self-dual Einstein four-manifolds are important as {\em gravitational instantons} in quantum theories of gravity. 
A classical reference is the monograph \cite{Besse2008} by Besse.
The role of these conditions in Theorem \ref{th:main} will be clarified by Theorems \ref{th:Atiyah} and \ref{th:Eells}.

The analogue of Theorem \ref{th:main} also holds for bordered Riemann surfaces with countably 
many boundary curves; see Theorem \ref{th:main2}. Every such surface is an open domain 
\begin{equation}\label{eq:countable}
	M=R\setminus \bigcup_{i=0}^\infty D_i
\end{equation}
in a compact Riemann surface $R$, where $D_i\subset R$ are pairwise disjoint smoothly 
bounded closed discs. By the uniformisation theorem of He and Schramm \cite{HeSchramm1993}, 
every open Riemann surface of finite genus and having at most countably many ends is conformally 
equivalent to a surface of the form \eqref{eq:countable}, where $D_i$ lift to round discs or points in the 
universal covering surface of $R$. This gives the following corollary to Theorems \ref{th:main} and \ref{th:main2}.

%
%
\begin{corollary}\label{cor:main}
Every self-dual or anti-self-dual Einstein four-manifold contains a complete conformally immersed
superminimal surface with Jordan boundary parameterised by any given bordered Riemann surface
with finitely or countably many boundary curves. 

In particular, every open Riemann surface of finite genus 
and having at most countably many ends, none of which are point ends, is conformally equivalent to a 
complete conformal superminimal surface in any self-dual or anti-self-dual Einstein four-manifold.
\end{corollary}

It is in general impossible to ensure completeness of a minimal surface 
at a point end unless $(X,g)$ is complete and the immersion $M\to X$ is proper at such end. 

The special case of Theorem \ref{th:main} when $X$ is the four-sphere $S^4$ is given by 
\cite[Corollary 1.10]{Forstneric2020Merg}; see also \cite[Theorem 7.5]{AlarconForstnericLarusson2019X}. 
Since the spherical metric is conformally flat, the Weyl tensor vanishes and 
Theorem \ref{th:main} applies to superminimal surfaces of both positive and negative spin in $S^4$.
The same holds for the hyperbolic $4$-space $H^4$; see Corollary \ref{cor:H4}. While $S^4$ 
admits plenty of supermininal surfaces of any given conformal type 
(see \cite[Corollary 7.3]{AlarconForstnericLarusson2019X}), every minimal 
surface in $H^4$ is uniformised by the disc $\D$ (see Corollary \ref{cor:H4}).

A natural question at this point is, how many Riemannian four-manifolds $(X,g)$ are there 
satisfying the conditions in Theorem \ref{th:main}? Among the 
complete ones with positive scalar curvature, there are not many. The classical Bonnet-Myers theorem 
(see Myers \cite{Myers1941} or do Carmo \cite[p.\ 200]{doCarmo1992}) states that if the Ricci curvature of an 
$n$-dimensional complete Riemannian manifold $(X,g)$ is bounded from below by a positive constant, 
then it has finite diameter and hence $X$ is compact. 
Further, a theorem of Friedrich and Kurke \cite{FriedrichKurke1982} from 1982
says that a compact self-dual Einstein four-manifold with positive scalar curvature is 
either isometric to $S^4$ or diffeomorphic to the complex projective plane $\CP^2$. 
Superminimal surfaces in $S^4$ and $\CP^2$ with their natural metrics have been studied extensively; 
see \cite{Bryant1982JDG,Gauduchon1987A,Gauduchon1987B,MontielUrbano1997,BoltonWoodward2000}.
Hitchin \cite{Hitchin1974} described in 1974 the topological type all four-dimensional compact 
self-dual Einstein manifolds with vanishing scalar curvature. He proved that such a space is either flat or a 
$K3$-surface, an Enriques surface, or the orbit space of an Enriques surface by an antiholomorphic involution. 
Conversely, it follows from the solution of the Calabi conjecture by S.-T.\ Yau \cite{Yau1977,Yau1978}
that every $K3$ surface admits a self-dual Einstein metric ($W^-=0$) with vanishing scalar curvature.
On the other hand, there are many self-dual Einstein manifolds with negative scalar curvature 
including all real and complex space forms. In particular, there is an infinite dimensional 
family of self-dual Einstein metrics with scalar curvature $-1$ on the unit ball $\B\subset \R^4$
having prescribed conformal structure of a suitable kind on the boundary sphere $S^3=b\B$;
see Graham and Lee \cite{GrahamLee1991}, 
Hitchin \cite{Hitchin1995}, and Biquard \cite{Biquard2002}.
Another construction of an infinite dimensional family of self-dual Einstein metrics 
was given by Donaldson and Fine \cite{DonaldsonFine2006} and Fine \cite{Fine2008}.
It was shown by Derdzinski \cite{Derdzinski1983} that a compact four-dimensional self-dual 
K\"ahler manifold is locally symmetric.

In the remainder of this introduction we outline the proof of Theorem \ref{th:main};
the details are given in Sect.\ \ref{sec:proof}. In sections \ref{sec:SM}--\ref{sec:twistor}
we provide a sufficiently complete account of the necessary ingredients from the theory of 
superminimal surfaces and twistor spaces to make the paper accessible to a wide audience. 
Several different definitions of superminimal surfaces are used in the literature, 
and hence statement which are formally the same need not be equivalent.  
We take care to present a coherent picture to an uninitiated reader 
with basic knowledge of complex analysis and Riemannian geometry. 

We shall use three key ingredients. The first two are provided by the {\em twistor theory} initiated by 
Penrose \cite{Penrose1967} in 1967. One of its main features from mathematical viewpoint is that it provides
harmonic maps from a given Riemann surface $M$ into a Riemannian four-manifold 
$(X,g)$ as projections of suitable holomorphic maps $M\to Z$ into the total space of the 
twistor bundle $\pi:Z\to X$. Although this idea is reminiscent of the Enneper-Weierstrass formula for 
minimal surfaces in flat Euclidean spaces (see Osserman \cite{Osserman1986}), it differs from it in certain key aspects. There are two twistor spaces $\pi^\pm : Z^\pm \to X$, 
reflecting the spin (see Sect.\ \ref{sec:twistor}). 
Their total spaces $Z^\pm$ carry natural almost complex structures $J^\pm$ (nonintegrable in general), 
and the fibres of $\pi^\pm$ are holomorphic rational curves in $Z^\pm$. 
The Levi-Civita  connection of $(X,g)$ determines a complex horizontal subbundle 
$\xi^\pm\subset TZ^\pm$ projecting by $d\pi^\pm$ isomorphically onto the 
tangent bundle of $X$. The key point of twistor theory pertaining to our paper is the 
{\em Bryant correspondence}; see Theorem \ref{th:Friedrich}.
This correspondence, discovered  by Bryant \cite{Bryant1982JDG} (1982) in the case when $X$ is the
four-sphere $S^4$ (whose twistor spaces $Z^\pm$ are the three dimensional complex projective space 
$\CP^3$, see Sect.\ \ref{sec:S4} for an elementary explanation), shows that superminimal surfaces in $X$ of
$\pm$ spin are precisely the projections of holomorphic horizontal curves in $Z^\pm$, i.e., 
curves tangent to the horizontal distribution $\xi^\pm$. 

The second ingredient is provided by a couple of classical integrability results.
According to Atiyah, Hitchin and Singer \cite[Theorem 4.1]{AtiyahHitchinSinger1978},  
the twistor space $(Z^\pm,J^\pm)$ of a smooth oriented Riemannian four-manifold $(X,g)$ is 
an integrable complex manifold if and only if the conformally invariant Weil tensor 
$W=W^++W^-$ of $g$  satisfies $W^+=0$ or $W^-=0$, respectively. Assuming that this 
holds, a result of Salamon \cite[Theorem 10.1]{Salamon1983} 
(see also Eells and Salamon \cite[Theorem 4.2]{EellsSalamon1985}) says that 
the horizontal bundle $\xi^\pm$ is a holomorphic hyperplane subbundle of $TZ^\pm$ if 
and only if $g$ is an Einstein metric, and in such case $\xi^\pm$ is a holomorphic contact bundle 
if and only if the scalar curvature of $g$ is nonzero. 

The third main ingredient is a recent result of Alarc\'on and the author 
\cite[Theorem 1.3]{AlarconForstneric2019IMRN} saying that holomorphic Legendrian immersions from 
bordered Riemann surfaces into any holomorphic contact manifold enjoy the Calabi-Yau property, i.e., 
the analogue of Theorem \ref{th:main} holds for such immersions. 
(See also \cite[Theorem 1.2]{AlarconForstnericLopez2017CM} for the standard complex contact structure 
on Euclidean spaces $\C^{2n+1}$, $n\ge 1$.) Analogous results hold for holomorphic immersions into 
any complex manifold of dimension $>1$, and for conformal minimal immersions into the flat Euclidean 
space $\R^n$ for any $n\ge 3$. We refer to the recent survey \cite{AlarconForstneric2019JAMS} for 
an account of these developments. The proof of Theorem \ref{th:main} is then completed and generalised 
to surfaces $M$ with countably many boundary curves in Sect.\ \ref{sec:proof}. 
In Sect.\ \ref{sec:S4} we take a closer look at the 
case when $X$ is the sphere $S^4$ or the hyperbolic space $H^4$.

%
%
%
%
\section{Superminimal surfaces in Riemannian four-manifolds}\label{sec:SM}

In this section we recall the notion of the indicatrix of a smooth surface in a smooth Riemannian 
four-manifold $(X,g)$ and the geometric definition of a superminimal surface. We follow the 
paper by Friedrich \cite{Friedrich1997} from 1997.  

Let $M\subset X$  be a smoothly embedded surface endowed with the induced metric. 
(Since our considerations in this section are local, they also apply to immersions $M\to X$.) 
The tangent bundle of $X$ splits along $M$ into the orthogonal direct sum 
$TX|_M=TM\oplus N$ where $N$ is the normal bundle of $M$ in $X$. 
Given a point $p\in M$ we let 
\[
	\Sym(T_pM)=\bigl\{A:T_pM\to T_pM: g(Au,v)=g(u,Av)\ \text{for all}\ u,v\in T_pM\bigr\}
\]
denote the three dimensional real vector space of linear symmetric self-maps of $T_pM$.
Fixing an orthonormal basis of $T_pM$, we identify $\Sym(T_pM)\cong \Sym(\R^2)$   
with the space of real symmetric $2\times 2$ matrices
and introduce the isometry $\Sym(T_pM) \stackrel{\cong}{\lra} \R^3$ by
\[ 
		\begin{pmatrix} a&b\\ b&c\end{pmatrix} \longmapsto 
		\left(\frac{a+c}{\sqrt 2},\sqrt{2} b, \frac{a-c}{\sqrt 2} \right).
\] 
Each unit normal vector $n\in N_p$, $|n|^2:=g(n,n)=1$, determines a second fundamental form
$S_p(n) : T_pM\to T_pM$ which belongs to $\Sym(T_pM)$. The unit normal vectors 
form a circle in the normal plane $N_p$ to $M$ at $p$, and the curve
\begin{equation}\label{eq:Ip2}
	I_p = \{S_p(n) : n\in N_p,\ |n|=1\} \subset  \Sym(T_pM) \cong\R^3 
\end{equation}
is called the {\em indicatrix} of $M$ at $p$. It was shown by Kommerell \cite{Kommerell1905}
that $I_p\subset \R^3$ is either a straight line segment which is symmetric around the origin $0\in \R^3$ 
(possibly reducing to $0$) or the intersection of a cylinder over an ellipse and a two plane. 
If $M$ is a minimal surface in $X$ then $I_p$ is a symmetric segment, 
an ellipse, or a circle; see Kommerell \cite{Kommerell1905} and Eisenhart \cite{Eisenhart1912}. 
For a fixed tangent vector $v\in T_pM$ we also consider the curve 
\begin{equation}\label{eq:Ipv}
	I_p(v)=\bigl\{S_p(n)v: n\in N_p,\ |n|=1\bigr\} \subset T_pM.
\end{equation}

%
%
\begin{definition}\label{def:SM2}
A smooth surface $M\subset X$ is {\em superminimal} if every curve $I_p(v)\subset T_pM$ 
$(p\in M,\ 0\ne v\in T_pM)$ is a circle with centre $0$ (which may reduce to the origin). 
The same definition applies to a conformally immersed surface $f:M\to X$. 
\end{definition}

\begin{remark}\label{rem:SM}
(A) A calculation in \cite[pp.\ 2-3]{Friedrich1997} shows that the indicatrix 
$I_p$ \eqref{eq:Ip2} of a superminimal surface $M\subset X$ at any point $p\in M$ 
is a circle in $\Sym(T_pM)\cong \R^3$ with centre $0$, and 
every superminimal surface is a minimal surface (see \cite[Proposition 3]{Friedrich1997}).
The converse fails in general, but see Remark \ref{rem:biquadratic} for some special cases.

(B) The above definition does not require orientability. 
If $M$ and $X$ are oriented, then we can introduce {\em superminimal surfaces of positive or negative
spin} by looking at the direction of rotation of the point $S_p(n)v\in I_p(v)\subset T_pM$ 
as the unit normal vector $n\in N_p$ traces the unit circle in a given direction.
This gives the two spaces $\SM^\pm(M,X)$ in Definition \ref{def:SM} which get interchanged
under the reversal of the orientation on $X$.

(C) The class of superminimal surfaces is invariant under isometries of $(X,g)$.
\qed\end{remark}

Superminimal surfaces have been studied by many authors; see 
in particular Kommerell \cite{Kommerell1905}, Eisenhart \cite{Eisenhart1912},  
Bor$\mathring{\textrm u}$vka \cite{Boruvka1928ii,Boruvka1928i}, Calabi \cite{Calabi1967}, 
Chern \cite{Chern1970A,Chern1970B}, Bryant \cite{Bryant1982JDG}, 
Friedrich \cite{Friedrich1984,Friedrich1997}, Eells and Salamon \cite{EellsSalamon1985}, 
Gauduchon \cite{Gauduchon1987A,Gauduchon1987B}, 
Wood \cite{Wood1987}, Montiel and Urbano \cite{MontielUrbano1997}, 
Bolton and Woodward \cite{BoltonWoodward2000}, Shen \cite{Shen1993,Shen1996}, 
and Baird and Wood \cite{BairdWood2003}. 
A recent contribution to the theory of superminimal surfaces in $S^4$ was made in
\cite[Sect.\ 7]{AlarconForstnericLarusson2019X}.

%
%
%
%
\section{Almost hermitian structures on $\R^4$ and quaternions}\label{sec:H}

In this section we recall some basic facts about linear almost hermitian structures on $\R^4$
and their representation by quaternionic multiplication. This material is standard 
(see e.g.\ \cite{AtiyahHitchinSinger1978,EellsSalamon1985}),
except for Lemma \ref{lem:complexstructures} which will be used in Sect.\ \ref{sec:S4}. 

Let $\langle\cdotp,\cdotp\rangle$ stand for the Euclidean inner product on $\R^4$. 
We denote  by $\Jscr^\pm(\R^4)$ the space of {\em almost hermitian structures on $\R^4$}, 
i.e., linear operators $J:\R^4\to\R^4$ satisfying the following three conditions:
\begin{enumerate}[\rm (a)]
\item $J^2=-\Id$,
\item $\langle Jx,Jy\rangle = \langle x, y\rangle$ for all $x,y\in\R^4$, and
\item letting $\omega(x,y)=\langle Jx,y\rangle$ denote the fundamental form of $J$, 
we have that $\omega\wedge\omega=\pm \Omega$ where $\Omega$ is the standard volume form 
on $\R^4$ with its canonical orientation.
\end{enumerate}
Condition (a) lets us identify $\R^4$ with $\C^2$ such that $J$ corresponds to the multiplication
by $\igot$ on $\C^2$; any such linear operator is called a (linear) {\em almost complex structure} on $\R^4$.
The second condition means that $J$ is compatible with the inner product on $\R^4$, hence
the word {\em almost hermitian}. The third condition specifies the orientation of $J$. Note that 
\[	
	\Jscr^+(\R^4)\cup \Jscr^-(\R^4)\subset SO(4).
\] 
Any choice of positively oriented orthonormal basis $e=(e_1,e_2,e_3,e_4)$ of $\R^4$ 
determines a pair of almost hermitian structures $J_e^\pm \in \Jscr^\pm(\R^4)$ by 
\begin{equation}\label{eq:Je}
	J^\pm_e (e_1)= e_2,\quad\ J^\pm_e (e_3)=\pm e_4.
\end{equation}
If $e'=(e'_1,e'_2,e'_3,e'_4)$ is another orthonormal basis in the same orientation class, 
there is a unique $A\in SO(4)$ mapping $e_i$ to $e'_i$ for $i=1,\ldots,4$, and hence 
\[ 
	J^\pm_e=A^{-1}\circ J_{e'}^\pm\circ A.
\] 
This shows that for any fixed $J\in \Jscr^+(\R^4)$, conjugation $A \mapsto A^{-1}\circ J\circ A$
by orthogonal rotations $A\in SO(4)$ acts transitively on $\Jscr^+(\R^4)$; the corresponding 
property also holds for $\Jscr^-(\R^4)$. 
The stabiliser of this action is the unitary group $U(2)$, the group of orthogonal rotations 
preserving the given structure $J$, and $\Jscr^\pm(\R^4)$ can be identified with the quotient 
$SO(4)/U(2)\cong S^2$. Conjugation by an element $A\in O(4)$ of the orthogonal group with $\det A=-1$
interchanges $\Jscr^+(\R^4)$ and $\Jscr^-(\R^4)$, and $O(4)/U(2)\cong \Jscr^+(\R^4)\cup \Jscr^-(\R^4)$. 
For instance, the two structures in \eqref{eq:Je} are interchanged by the 
orientation reversing map $A\in O(4)$ given by $Ae_1=e_1,\ Ae_2=e_2,\ Ae_3=e_4,\ Ae_4=e_3$.
Note however that the structures $\pm J$ belong to the same space $\Jscr^\pm(\R^4)$. 

It is classical that every $A\in SO(4)$ is represented by a pair of 
rotations for angles $\alpha,\beta\in (-\pi,+\pi]$ in orthogonal cooriented $2$-planes 
$\Sigma\oplus \Sigma^\perp=\R^4$. (Such pair of planes is uniquely determined by $A$ if and only if 
$|\alpha|\ne |\beta|$.) The rotation $A$ is said to be {\em left isoclinic} if $\alpha=\beta$ 
(it rotates for the same angle in the same direction on both planes), and {\em right isoclinic} if
$\alpha=-\beta$ (it rotates for the same angle but in the opposite directions).
Thus, elements of $\Jscr^+(\R^4)$ are precisely the left isoclinic rotations for the angle $\pi/2$,
while those in $\Jscr^-(\R^4)$ are the right isoclinic rotations for the angle $\pi/2$.

%
%
Here is another interpretation of the spaces $\Jscr^\pm(\R^4)$;
see Atiyah et al.\ \cite[Sect.\ 1]{AtiyahHitchinSinger1978} or Eells and Salamon 
\cite[Sect.\ 2]{EellsSalamon1985}. Let $\Lambda^2(\R^4)$ denote the second exterior
power of $\R^4$. For any oriented orthonormal basis $e_1,\ldots, e_4$ of $\R^4$
the vectors $e_i\wedge e_j$ for $1\le i<j\le 4$ form an orthonormal basis of $\Lambda^2(\R^4)$,
so $\dim_\R \Lambda^2(\R^4)=6$. The Hodge star endomorphism  $*:\Lambda^2(\R^4)\to \Lambda^2(\R^4)$ 
is defined by $\alpha\wedge *\beta=\langle \alpha,\beta\rangle \Omega \in \Lambda^4(\R^4)$. 
We have that $*^2=1$, and the $\pm1$ eigenspace $\Lambda^2_\pm(\R^4)$ of $*$ 
has an oriented orthonormal basis
\begin{equation}\label{eq:basisofLambda2}
	e_1\wedge e_2 \pm e_3\wedge e_4,\quad e_1\wedge e_3 \pm e_4\wedge e_2,
	\quad e_1\wedge e_4 \pm e_2\wedge e_3.
\end{equation}
The Euclidean metric lets us identify $\R^4$ with its dual $(\R^4)^*$, which gives the inclusion 
\begin{equation}\label{eq:LEnd}
	\Lambda^2(\R^4) \hra \R^4\otimes \R^4\cong (\R^4)^*\otimes  \R^4 = \End(\R^4) \cong GL_4(\R).
\end{equation}
Under this identification of $\Lambda^2(\R^4)$ with a subset of $\End(\R^4)$, we have that
\begin{equation}\label{eq:Spm}
	\Jscr^\pm(\R^4) = S(\Lambda^2_\pm(\R^4)) := 
	\text{the unit sphere of}\ \Lambda^2_\pm(\R^4)\cong\R^3.
\end{equation}
For example, the vector $e=e_1\wedge e_2 + e_3\wedge e_4\in \Lambda^2_+(\R^4)$ 
is sent under the first inclusion in \eqref{eq:LEnd} to 
$e_1\otimes e_2-e_2\otimes e_1+e_3\otimes e_4-e_4\otimes e_3\in \R^4\otimes \R^4$, 
and under the second isomorphism in \eqref{eq:LEnd} to the almost hermitian structure 
given by \eqref{eq:Je}:
\[
	J_e=e^*_1\otimes e_2-e^*_2\otimes e_1+e^*_3\otimes e_4-e^*_4\otimes e_3 \in \Jscr^+(\R^4).
\]
We adopt the following convention regarding the orientations. (It is difficult 
to find this essential point in the construction of twistor spaces spelled out in the literature.)

%
%
\noindent \bf Orientation on $\Jscr^\pm(\R^4)$. \label{p:convention} \rm 
Let $e=(e_1,e_2,e_3,e_4)$ be a positively oriented orthonormal basis of $\R^4$, and  
let the spaces $\Lambda^2_\pm(\R^4)\cong\R^3$ be oriented by the pair of bases 
\eqref{eq:basisofLambda2}. We endow $\Jscr^+(\R^4)= S(\Lambda^2_+(\R^4))$ 
with the {\em outward orientation} of the unit $2$-sphere in $\Lambda^2_+(\R^4)\cong\R^3$, while 
$\Jscr^-(\R^4)= S(\Lambda^2_-(\R^4))$ is given the {\em inward orientation}.

Letting $\overline\R^4$ denote $\R^4$ with the opposite orientation, 
it is easily checked that we have {\em orientation preserving} isometric isomorphisms  
\[ 
	\Jscr^\pm(\R^4)=S(\Lambda^2_\pm(\R^4)) \ \stackrel{\cong}{\lra}\ 
	S(\Lambda^2_\mp(\overline\R^4))=\Jscr^\mp(\overline\R^4).
\] 

An oriented $2$-plane $\Sigma\subset \R^4$ determines a pair of almost hermitian structures
$J^\pm_\Sigma \in \Jscr^\pm(\R^4)$ which rotate for $\pi/2$ in the positive direction 
on $\Sigma$ and for $\pm \pi/2$ on its cooriented orthogonal complement $\Sigma^\perp$.
Denoting by $G_2(\R^4)$ the Grassmann manifold of oriented $2$-planes 
in $\R^4$, we have that (cf.\ \cite[p.\ 595]{EellsSalamon1985}) 
\begin{equation}\label{eq:G2}
		G_2(\R^4) \cong S(\Lambda^2_+(\R^4)) \times S(\Lambda^2_-(\R^4)) 
		= \Jscr^+(\R^4) \times \Jscr^-(\R^4).
\end{equation}

%
%
Almost hermitian structures on $\R^4$ can be represented by quaternionic multiplication.
Let $\H$ denote the field of quaternions. An element of $\H$ is written uniquely as
\begin{equation}\label{eq:qx}
	q= x_1+x_2\igot  + x_3\jgot  + x_4 \kgot  = z_1+z_2\jgot, 
\end{equation}
where $(x_1,x_2,x_3,x_4)\in\R^4$, $z_1=x_1+x_2 \igot \in\C,\ z_2=x_3+x_4\igot \in \C$, and  
$\igot,\jgot,\kgot$ are the quaternionic units satisfying 
\[
	\igot^2=\jgot^2=\kgot^2=-1,\quad \igot\jgot=-\jgot\igot=\kgot,\quad
	\jgot\kgot=-\kgot\jgot=\igot,\quad \kgot\igot=-\igot\kgot=\jgot.
\]
We identify $\R^4$ with $\H$ using $1,\igot,\jgot,\kgot$ as the standard positively oriented orthonormal basis.
(Some authors write complex coefficients on the right in \eqref{eq:qx}; due to noncommutativity 
this makes for certain differences in the constructions  and formulas.) Recall that 
\[
	\bar q=x_1-x_2\igot  - x_3\jgot  - x_4 \kgot,\quad 
	q\bar q=|q|^2=\sum_{i=1}^4 x_i^2, \quad 
	q^{-1}=\frac{\bar q}{|q|^2}\ \text{if}\ q \ne 0,\quad \overline{pq}=\bar q \bar p. 
\]
By $\H_0$ we denote the real $3$-dimensional subspace of purely imaginary quaternions:
\begin{equation}\label{eq:H0}
	\H_0=\{q=x_2\igot  + x_3\jgot  + x_4 \kgot: x_2,x_3,x_4\in \R\} \cong \R^3.
\end{equation}
We also introduce the spheres of unit quaternions and imaginary unit quaternions:
\begin{equation}\label{eq:spheres}	
	\Scal^3:=\{q\in \H:|q|=1\} \cong S^3,\qquad 	\Scal^2=\{q\in \H_0:|q|=1\}\cong S^2. 
\end{equation}
We take $\igot,\jgot,\kgot$ as a positive orthonormal basis of $\H_0$ and orient 
the spheres $\Scal^2\subset \H_0$ and $\Scal^3\subset \H$ by the respective outward 
normal vector field. In particular, the vectors $\jgot,\kgot$ are a positively oriented orthonormal 
basis of the tangent space $T_\igot \Scal^2$. 

Elements of $\Jscr^+(\R^4)$ and $\Jscr^-(\R^4)$ then correspond to left and right multiplications, 
respectively, on $\H\cong \R^4$ by imaginary unit quaternions $q\in \Scal^2$. 
To see this, note that every $J\in \Jscr^+(\R^4)$ is uniquely 
determined by its value $q=J(1)$ on the first basis vector; this value is orthogonal to 
$1$ and of unit length, hence an element of the unit sphere $\Scal^2\subset \H_0$ inside the 
$3$-space of imaginary quaternions \eqref{eq:spheres}.
The pair $1,q$ spans a $2$-plane $\Sigma\subset\H$ whose orthogonal complement 
$\Sigma^\perp$ is contained in the hyperplane $\H_0$. 
The left multiplication by $q$ on $\H$ then amounts to a rotation for $\pi/2$
in the positive direction on $\Sigma^\perp$, while the right multiplication by $q$ yields 
a rotation for $\pi/2$ in the negative direction on $\Sigma^\perp$.
The left multiplication by $\igot$ determines the {\em standard structure} 
$J_\igot(1)=\igot$, $J_\igot(\jgot)=\kgot$.

The following lemma will be used in Sect.\ \ref{sec:S4} to provide an elementary explanation of the fact 
that $\CP^3$ is the twistor space of $S^4$. The analogous result holds for $\Jscr^-(\R^4)$ as seen 
by using the right multiplication on $\H$ by nonzero quaternions.

%
%
\begin{lemma}\label{lem:complexstructures}
For every $q\in \H\setminus \{0\}$ the left multiplication by $\bar q$ on $\H$
uniquely determines an almost hermitian structure $J_q\in \Jscr^+(\R^4)$ making the 
following diagram commute:
\[ 
	\xymatrix{
	\R^4 \ar[r]^{\cong}\ar[d]_{J_\igot} & \H  \ar[r]^{\bar q\,\cdotp} \ar[d]_{\igot\,\cdotp} 
	& \H \ar[d]^{\bar q \igot \bar q^{-1}\cdotp} \ar[r]^{\cong} & \R^4 \ar[d]^{J_q}  \\ 
	\R^4 \ar[r]^{\cong} & \H \ar[r]^{\bar q\,\cdotp} & \H \ar[r]^{\cong} & \R^4 
	} 
\]
The map $\H\setminus \{0\} \to \Jscr^+(\R^4)$ given by $q\mapsto J_q$ is 
equivalent to the canonical projection $\H\setminus \{0\} = \C^2_*\to\CP^1$ under an 
orientation preserving diffeomorphism $\Jscr^+(\R^4)\to \CP^1$. 
\end{lemma}

\begin{proof}
From $q\bar q=|q|^2$ we see that $\bar q^{-1}=q/|q|^2$ and hence 
\[	
	\bar q \, \igot \bar q^{-1}=\frac{\bar q}{|q|^2}\igot q=q^{-1}\igot q \in \Scal^3.
\] 
For any $q_1,q_2\in\H$ we have that $\overline{q_1q_2}=\bar q_2\bar q_1$ and hence 
\[
	\overline{\bar q \, \igot \bar q^{-1}} = q^{-1} (-\imath) q = - q^{-1} \, \imath q,
\]
so $q^{-1}\igot q \in \Scal^2$ is a purely imaginary unit quaternion. It follows that the left product by
$q^{-1}\igot q$ on $\H$ determines an almost hermitian structure $J_q\in \Jscr^+(\R^4)$. 

Let us consider more closely the map 
\[ 
	\Phi:\H\setminus \{0\}\to \Scal^2,\qquad \Phi(q)=q^{-1} \, \imath q.
\] 
We have that $\Phi(q_1)=\Phi(q_2)$ if and only if
\[
	q_1^{-1}\imath q_1=q_2^{-1}\imath q_2 \ \Longleftrightarrow \ 
	(q_2q_1^{-1}) \imath = \imath (q_2 q_1^{-1}) \ \Longleftrightarrow\ q_2 q_1^{-1}\in \C^*,
\]
so the fibres of $\Phi$ are the punctured complex lines $\C^*q$ for $q\in \H\setminus\{0\}$. 

We claim that $\Phi$ is a submersion. Since $\Phi$ is constant on the  lines $\C^*q$,
it suffices to show that $\Phi:\Scal^3\to \Scal^2$ is a submersion.  Fix $q\in \Scal^3$.
For any $q'\in\H$ we have that 
\[
	d\Phi_q(q')=\frac{d}{dt}\Big|_{t=0} \Phi(q+tq')=
	\frac{d}{dt}\Big|_{t=0} \overline{(q+tq')}\, \igot (q+tq') 
	= \bar q' \igot q + \bar q\, \igot q'.
\]
In particular, 
\[ 
	d\Phi_q(\jgot q)= 2\bar q\, \kgot q,\qquad d\Phi_q(\kgot q)= -2\bar q\, \jgot q.
\] 
These two vector are clearly $\R$-linearly independent, so $d\Phi_q:T_q\Scal^3\to T_{\Phi(q)}\Scal^2$ 
has rank $2$ at each point. For $q=\igot$ we get that $\Phi(\imath)=\imath$ and $d\Phi_\igot(\jgot) = 2\jgot$, $d\Phi_\igot(\kgot) = 2\kgot$. Note that $(\jgot,\kgot)$ is a positively oriented orthonormal basis 
of both $T_\igot \Scal^2$ and $T_{[1:0]}\CP^1$, the tangent space at the point $[1:0]$ to the 
projective line consisting of complex lines in $\H=\C^2$, with $[1:0]=\C\times \{0\}$. 
It follows that $\Phi=h\circ \phi$ where $\phi:\C^2_*\to\CP^1$ is the canonical projection and 
$h:\CP^1\to \Scal^2$ is an injective orientation preserving local diffeomorphism, hence an
orientation preserving diffeomorphism onto $\Scal^2$. (Surjectivity is easily seen by an explicit calculation.)
Finally, we identify $\Jscr^+ (\R^4)$ with $\Scal^2$ acting on $\R^4=\H$ by left multiplication;
this identification is orientation preserving as well.  

Note that the map $\Phi:S^3\cong \Scal^3 \to\Scal^2\cong S^2$ is the Hopf fibration
with circle fibres $\{\E^{\imath t}q:t\in\R\}\cong S^1$, $q\in \Scal^3$. 
\end{proof}

%
%
%
%
\section{Twistor bundles and the Bryant correspondence}\label{sec:twistor}

In 1967, Penrose \cite{Penrose1967} introduced a new {\em twistor theory} with an immediate
goal of studying representation theory of the $15$-parameter Lie group of conformal coordinate
transformations on four-dimensional Minkowski space leaving the light-cone invariant. 
(The mathematical ideas in Penrose's paper are in close relation to those developed in the 
notes \cite{Veblen1955} of the seminar conducted by Oswald Veblen and John von Neumann during 1935--1936.) 
One of his aims was to offer a possible path to understand quantum gravity; 
see Penrose and MacCallum \cite{PenroseMacCallum1973}. Penrose also promoted the idea that twistor spaces 
should be the basic arena for physics from which space-time itself should emerge. 

Mathematically, twistor theory connects four-dimensional Riemannian geometry to
three-dimensional complex analysis. A basic example is the complex projective three-space $\CP^3$ 
as the twistor space of $S^4$ with the spherical metric (see Penrose \cite[Sect.\ VI]{Penrose1967},
Bryant \cite{Bryant1982JDG}, and Sect.\  \ref{sec:S4}). Physically it 
is the space of massless particles with spin. 
Twistor theory evolved into a branch of mathematics and theoretical physics  
with applications to differential and integral geometry, nonlinear differential equations and 
representation theory and in physics to relativity and quantum field theory.
%
%
For the theory of twistor spaces, see in particular the 
papers by Atiyah, Hitchin and Singer \cite{AtiyahHitchinSinger1978}, Friedrich \cite{Friedrich1984},
Eells and Salamon \cite{EellsSalamon1985}, Gauduchon \cite{Gauduchon1987A,Gauduchon1987B},
the monographs by Ward and Wells \cite{WardWells1990} and Baird and Wood \cite{BairdWood2003},
and the recent survey by Sergeev \cite{Sergeev2018}. Twistor theory also exists for certain Riemannian 
manifolds of real dimension $4n$ for $n>1$, in particular for quaternion-K\"ahler manifolds
(see Salamon \cite{Salamon1982}, LeBrun and Salamon \cite{LeBrunSalamon1994IM}, 
and LeBrun \cite{LeBrun1995IJM}).

Associated to an oriented Riemannian four-manifold $(X,g)$ is a pair of almost hermitian fibre bundles
with fibre $\CP^1$, 
\[ 	
	Z^\pm(X) := \Jscr^\pm (TX)=S(\Lambda^2_\pm(TX))\ \stackrel{\pi^\pm}{\lra} X,
\] 
the {\em positive} and the {\em negative twistor bundle} of $X$. The fibre over any point $x\in X$ equals
\[
	(\pi^\pm)^{-1}(x) = \Jscr^\pm(T_x X) = S(\Lambda^2_\pm(T_x X)) \cong \CP^1,
\] 
the space of positive or negative almost hermitian structures on $T_x X\cong \R^4$.
(The second equality uses the identification \eqref{eq:Spm}.) 
The complex structure on $\Jscr^\pm(T_x X)\cong S^2$ is specified by the choice of orientation 
in Sect.\ \ref{sec:H}, p.\ \pageref{p:convention}.
A local trivialisation of $Z^\pm\to X$ is provided by an oriented orthonormal
frame field $e(x)=(e_1(x),\ldots,e_4(x))$ for $TX$ on an open set $x\in U\subset X$.
If $e'(x)$ is another such frame field on $U'\subset X$ then the transition map 
between the associated fibre bundle charts is given by conjugation with the 
field of linear maps $A(x)\in SO(T_x X)\cong SO(\R^4)$ sending $e(x)$ to $e'(x)$ for $x\in U\cap U'$. 

The Levi-Civita connection associated to the metric $g$ on $X$ induces at any point $z\in Z^\pm$ a 
decomposition of the tangent space $T_z Z^\pm$ into the direct sum 
\[
	T_z Z^\pm = T_z^h Z^\pm \oplus T_z^v Z^\pm = \xi_z^\pm \oplus T_z^v Z^\pm,
\]
where $T_z^v Z^\pm= T_z \pi^{-1}(\pi(z))$ is the vertical tangent space (the tangent space to the fibre)
and  $\xi_z^\pm = T_z^h Z^\pm$  is the {\em horizontal space}.
This defines a {\em horizontal subbundle} $\xi^\pm \subset TZ^\pm$ such that 
the differential $d\pi^\pm_z: \xi_z^\pm \to T_{\pi^\pm(z)}X$ is an isomorphism for each $z\in Z^\pm$.
Every path $\gamma(t)$ in $X$ with $\gamma(0)=x$ admits a unique horizontal lift
$\lambda(t)$ in $Z^\pm$ (tangent to $\xi^\pm$)
with any given initial point $\lambda(0)=z\in (\pi^\pm)^{-1}(x) = \Jscr^\pm(T_xX)$,
obtained by the parallel transport of $z$ along $\gamma$ with respect to the Levi-Civita connection.
However, lifting a surface in $X$ to a horizontal surface is $Z^\pm$ 
is in general impossible due to noninvolutivity of $\xi^\pm$.

There is a natural almost complex structure $J^\pm$ on $Z^\pm$ determined by the 
condition that at each point $z\in Z$, $J^\pm_z$ agrees with the standard
almost complex structure on the vertical space $T_z^v Z^\pm \cong T_z\CP^1$, 
while on the horizontal space $\xi_z^\pm$ we have that
\begin{equation}\label{eq:twistor}
	d\pi^\pm_z \circ J^\pm_z = z\circ d \pi^\pm_z.
\end{equation}
It follows that $\xi^\pm$ is a $J^\pm$-complex subbundle of the tangent bundle $TZ^\pm$.
(The structure $J^\pm$ introduced above is denoted $J_1$ in 
\cite{AtiyahHitchinSinger1978,EellsSalamon1985}; the second structure 
$J_2^\pm$ is obtained by reversing the orientations on the fibres of twistor projections.
As shown in \cite{Salamon1985}, the structure $J_2^\pm$ is never integrable, but is nevertheless
interesting in view of \cite[Theorem 5.3]{EellsSalamon1985}.)

%
%
Here is a summary of some basic properties of twistor bundles. 

\begin{proposition}\label{prop:summary}
\begin{enumerate}[\rm (a)]
\item
Denoting by $\overline X$ the Riemannian manifold $X$ endowed with the same metric and
the opposite orientation, we have that
\[ 
	Z^+(\overline X) = Z^-(X),\quad\ Z^-(\overline X)=Z^+(X)
\] 
as hermitian fibre bundles over $X$, and also as almost complex manifolds. In particular, 
their horizontal bundles and the respective almost complex structures on them agree.
\item
There are antiholomorphic involutions $\iota^\pm :Z^\pm \to Z^\pm$ preserving the fibres of 
$\pi^\pm :Z^\pm \to X^\pm$ and taking any $J \in \Jscr^\pm(T_x X)$ to $-J\in \Jscr^\pm(T_x X)$. 
(Identifying the fibre with $\CP^1$, this is the map $z\mapsto -1/\bar z$ on each fibre.)
\item An orientation preserving isometry $\phi:X\to X$ lifts to holomorphic 
isometries $\Phi^\pm :Z^\pm\to Z^\pm$ preserving $\xi^\pm$ such that 
$\pi^\pm\circ \Phi^\pm =\phi\circ \pi^\pm$. Moreover, the almost complex type of $(Z^\pm,J^\pm)$   
only depends on the conformal class of a metric on $X$, but the horizontal spaces
$\xi^\pm$ depend on the choice of metric in that class.
\item An orientation reversing isometry $\theta:X\to X$ lifts to a holomorphic isometry 
$\Theta:(Z^+(X),J^+)\to (Z^-(X),J^-)$ making the following diagram commute:
\begin{equation}\label{eq:reversing}
	\xymatrix{
	Z^+(X) \ar[r]^{\Theta}\ar[d]_{\pi^+} & Z^-(X) \ar[d]^{\pi^-} \\ 
	X \ar[r]^{\theta} & X
	} 
\end{equation}
\end{enumerate}
\end{proposition}

An example of (d) is the antipodal map on $X=S^4$, and in this
case $Z^+(S^4)\cong Z^-(S^4)\cong \CP^3$ (see Bryant \cite{Bryant1982JDG},
Gauduchon \cite{Gauduchon1987B}, and Sect.\ \ref{sec:S4}).

%
%
\begin{example}\label{ex:R4}
(A) The twistor bundle $Z^+$ of $\R^4$ with the Euclidean metric is fibrewise diffeomorphic 
to $\R^4\times \CP^1$, and its horizontal distribution $\xi$ is involutive with the leaves 
$\R^4\times \{z\}$ for $z\in \CP^1$.
The almost complex structure $J^+$ on $Z^+$ restricted to the leaf $L_z=\R^4 \times \{z\}$ equals
$z\in \Jscr^+(\R^4)$, and $(L_z,z)$ is a complex manifold which is biholomorphic to $\C^2$ under
a rotation in $SO(4)$. As a complex manifold, $(Z^+,J^+)$ is biholomorphic to the total space of the 
vector bundle $\Oscr(1)\oplus \Oscr(1) \to \CP^1$, and the leaves $L_z$ of $\xi$ are the fibres of 
this projection. See \cite[Remark 2, p.\ 266]{Friedrich1984} for more details.
\qed \end{example}

%
%
Recall from \eqref{eq:G2} that an oriented $2$-plane $\Sigma\subset T_xX$
determines a pair of almost hermitian structures $J^\pm_\Sigma \in \Jscr^\pm(T_x X)$.
Let $M$ be an oriented surface. To any immersion 
$f:M\to X$ we associate the {\em twistor lifts} $F^\pm :M\to Z^\pm$ 
with $\pi^\pm\circ F^\pm =f$ by the condition that for any point $p\in M$ and $x=f(p)\in X$, 
\begin{equation}\label{eq:twistorlift}
	\text{$F^\pm(p) \in \Jscr^\pm(T_x X)$ is determined by the 
	oriented $2$-plane $df_p(T_p M)\subset T_xX$}.
\end{equation}
That is, $F^\pm(p)$ rotates for $+\pi/2$ in the oriented plane $\Sigma=df_p(T_p M)$
and for $\pm \pi/2$ in the cooriented orthogonal plane $\Sigma^\perp$. 
\vspace{-3mm}
\[	
	\xymatrix{& Z^\pm \ar[d]^{\pi^\pm} \\
	M \ar[r]^{f} \ar[ur]^{F^\pm} & X
	}
\]
Here is a more explicit description. Assume for simplicity that $M\subset X$ is embedded and let 
$TX|_M=TM\oplus N$ where $N$ is the orthogonal normal bundle of $M$ in $X$. 
Locally near any point $p\in M$ 
there is an oriented orthonormal frame field $(e_1,e_2,e_3,e_4)$ for $TX$ such that, along $M$, 
$(e_1,e_2)$ is an oriented frame for $TM$ while $(e_3,e_4)$ is a frame for $N$. Then,
$F^\pm$ is determined by the conditions $F^\pm e_1=e_2,\ \ F^\pm e_3=\pm e_4$. 

%
%
\begin{remark}\label{rem:regularity}
(A) The twistor lifts $F^\pm$ clearly depend on the first order jet of $f$.
Hence, if the immersion $f:M\to X$ is of class $\Cscr^r$ $(r\ge 1)$ then 
$F^\pm:M\to Z^\pm$ are of class $\Cscr^{r-1}$. 

(B) If $\wt M$ is the Riemann surface $M$ with the opposite orientation and 
$\wt F^\pm:\wt M\to Z^\pm$ denote the respective twistor lifts of $f:M\to X$, 
then $\wt F^\pm  = \iota^\pm \circ F^\pm$ where $\iota^\pm$ is the antiholomorphic involution 
on $Z^\pm$ in Proposition \ref{prop:summary} (B).

(C) An orientation reversing isometry $\theta:X\to X$ maps every superminimal surface
$f:M\to X$ of $\pm$ spin to a superminimal surface $\theta\circ f:M\to X$ of $\mp$ spin.
\qed \end{remark}

We have the following additional properties of twistor lifts of a {\em conformal} immersion.
The second statement is the first part of \cite[Proposition 3]{Friedrich1984}; note 
however that in \cite{Friedrich1984} an immersion $f:M\to X$ is tacitly assumed to be conformal.

%
%
\begin{lemma}\label{lem:conformal}
If $I$ is an almost complex structure on $M$ and $f:(M,I)\to X$ is a conformal immersion, 
then $F^\pm(p)\in \Jscr^\pm(T_{f(p)}X)$ $(p\in M)$ is uniquely determined by the condition
\begin{equation}\label{eq:lift2}
	df_p\circ I_p = F^\pm(p)\circ df_p.
\end{equation}
Furthermore, the horizontal part $(dF^\pm_p)^h$ of the differential of $F^\pm$ satisfies
\begin{equation}\label{eq:h-holo}
	(dF^\pm_p)^h \circ I_p = J^\pm_{F(p)}\circ (dF^\pm_p)^h,\quad p\in M.
\end{equation}
In particular, if the twistor lift $F^\pm$ of a conformal immersion $f:M\to X$ 
is horizontal, then it is holomorphic as a map from $(M,I)$ into $(Z^\pm,J^\pm)$.
\end{lemma}

\begin{proof}
The formula \eqref{eq:lift2} is an immediate consequence of the definition of $F^\pm$ and the 
conformality of $f$. Let $F$ denote any of the lifts $F^\pm$. From $\pi\circ F=f$ we get that
\begin{equation}\label{eq:chain}
	d\pi_{F(p)}\circ dF^h_p = d\pi_{F(p)}\circ dF_p =df_p,\quad p\in M, 
\end{equation}
and hence
\[ 
	d\pi_{F(p)}\!\circ\! dF^h_p \!\circ\! I_p \stackrel{\eqref{eq:chain}}{=}
	df_p\!\circ\! I_p \stackrel{\eqref{eq:lift2}}{=} 
	F(p)\!\circ\! df_p \stackrel{\eqref{eq:chain}}{=} F(p)\!\circ\! d\pi_{F(p)}\!\circ\! dF^h_p
	\stackrel{\eqref{eq:twistor}}{=} d\pi_{F(p)}\!\circ\! J^\pm_{F(p)}\!\circ\! dF_p^h.
\]
Since the vectors under $d\pi_{F(p)}$ are horizontal, \eqref{eq:h-holo} follows.
\end{proof}

We now consider conformal immersions $f:M\to X$ which arise as projections 
to $X$ of holomorphic immersions $F:M\to Z^\pm$. 
The following result is \cite[Proposition 1]{Friedrich1984}. 

%
%
\begin{lemma}\label{lem:P1}
Let $(Z,J)$ denote any of the twistor manifolds $(Z^\pm(X),J^\pm)$. 
If $F:(M,I) \to (Z,J)$ is a holomorphic immersion such that $dF_p(T_p M)$ intersects the 
vertical tangent space $T^v_{F(p)}Z$ only at $0$ for every $p\in M$, 
then $F$ agrees with the twistor lift $F^\pm$ \eqref{eq:twistorlift} of its projection $f=\pi\circ F:M\to X$.
\end{lemma}

\begin{proof}
The conditions on $F$ implies that $f$ is an immersion. Fix a point $p\in M$. 
Since $F$ is holomorphic and the horizontal space $T^h_{F(p)}Z$ in $J$-invariant, 
\eqref{eq:h-holo} holds and hence 
\[
	df_p\circ I_p \stackrel{\eqref{eq:chain}}{=}  
	 d\pi_{F(p)} \circ dF^h_p \circ I_p \stackrel{\eqref{eq:h-holo}}{=} d\pi_{F(p)} \circ J_{F(p)}\circ dF^h_p
	\stackrel{\eqref{eq:twistor}}{=} F(p)\circ d\pi_{F(p)} \circ dF^h_p 
	\stackrel{\eqref{eq:chain}}{=}  F(p) \circ df_p.
\]
This shows that $f$ is conformal and $F$ is its twistor lift (cf.\ \eqref{eq:lift2}).
\end{proof}

The following key statement combines the above observations with \cite[Proposition 4]{Friedrich1984}. 
When $X=S^4$ with the spherical metric, this is due to Bryant \cite[Theorems B, B']{Bryant1982JDG};
the general case was proved by Friedrich \cite[Proposition 4]{Friedrich1984}. 

%
%
\begin{theorem}[The Bryant correspondence]
\label{th:Friedrich}
Let $M$ be a Riemann surface, and let $(X,g)$ be an oriented Riemannian four-manifold. The 
following conditions are pairwise equivalent for a smooth conformal immersion $f:M\to X$
(with the same choice of $\pm$ in every item).
\begin{enumerate}[\rm (a)] 
\item $f$ is superminimal of $\pm$ spin (see Definition \ref{def:SM}).
\item $f$ admits a holomorphic horizontal lift $M\to Z^\pm(X)$.
\item The twistor lift $F^\pm:M\to Z^\pm(X)$ of $f$ (see \eqref{eq:twistorlift}) is horizontal.
\item We have that $\nabla F^\pm=0$, where $\nabla$ is the covariant derivative 
on the vector bundle $f^*(TX)\to M$ induced by the Levi-Civita connection on $X$.
\end{enumerate}
\end{theorem}

\begin{proof}[Sketch of proof] 
The equivalence of (b) and (c) follows from Lemma \ref{lem:P1}.

Consider now (a)$\Leftrightarrow$(c). In \cite[Proposition 4]{Friedrich1984}, horizontality of the twistor lift 
$F^-:M\to Z^-$ (condition (c)) is characterized by a certain geometric property of the second 
fundamental forms $S_p(n):T_pM\to T_pM$ of $f$ at $p\in M$ in unit normal directions $n\in N_p$.
An inspection of the proof shows that this property is equivalent to $f$ being a superminimal 
surface of negative spin in the sense of Definition \ref{def:SM}, hence to condition (a). 
Although not stated in \cite{Friedrich1984}, the same proof gives the 
analogous conclusion for conformal superminimal immersions $f:M\to X$
of positive spin with respect to the twistor lift $F^+:M\to Z^+$. 
The crux of the matter can be seen from the display 
on the middle of page 266 in \cite{Friedrich1984} which shows that the rotation of
the unit normal vector $n\in N_pM$ in a given direction corresponds to the rotation of the point 
$S_p(n)v\in I_p(v)\subset T_pM$ \eqref{eq:Ixv} in the opposite direction 
(assuming that the spaces $T_pM$ and $N_pM$ are coorriented). 
Reversing the orientation on $X$, $F^-$ is replaced by $F^+$ and the respective curves now rotate 
in the same direction, so $F^+$ is horizontal if and only if $f$ is superminimal of positive spin. 
The direction of rotation is irrelevant (only) at points $p\in  M$ where the normal curvature 
of the immersion $f$ vanishes and hence the circle $I_p(v)$ reduces to the origin.

Concerning (c)$\Leftrightarrow$(d), Friedrich showed in \cite[Proposition 5, p.\ 270]{Friedrich1984} 
that the twistor lift $F^-$ is horizontal if and only if the immersion $f$ is {\em negatively oriented-isoclinic}. 
It is immediate from his description that the latter property simply says that the almost complex structure 
on the vector bundle $f^*TX=TM\oplus N$ adapted to $f$ (which is precisely the structure $F^-$) 
is invariant under parallel transport along curves in $M$; equivalently, $F^-$  is parallel with respect to 
the covariant derivative $\nabla $ on $f^*TX$ induced by the Levi-Civita connection on $X$: 
$\nabla F^-=0$. Reversing the orientation on $X$, the analogous conclusion shows that $F^+$ is 
horizontal if and only if $f$ is {\em positively oriented-isoclinic} if and only if $\nabla F^+=0$.
(See also \cite[Proposition 17]{Gauduchon1987A} and \cite[Proposition 1]{MontielUrbano1997}.)
\end{proof}

In light of the Bryant correspondence, it is a natural question whether not necessarily horizontal
holomorphic curves in twistor spaces $Z^\pm(X)$ might yield a larger class or minimal surfaces
in the given Riemannian four-manifold $X$. In fact, this is not so as shown by the following result
of Friedrich \cite[Proposition 3]{Friedrich1984}. (Note that in \cite{Friedrich1984} 
an immersion $f:M\to X$ is called superminimal if and only if its twistor lift is horizontal.)

%
%
\begin{lemma}\label{lem:P3}
The following are equivalent for a smooth conformal immersion $f:M\to X$.
\begin{enumerate}[\rm (i)]
\item The twistor lift $F^\pm:M\to Z^\pm$ of $f$ is horizontal. (By Theorem \ref{th:Friedrich}, this is
equivalent to saying that $f$ is superminimal of $\pm$ spin.) 
\item $f$ is a minimal surface in $X$ and it admits a holomorphic lift $\tilde f:M\to Z^\pm$.
\end{enumerate}
\end{lemma}

\begin{proof}[Sketch of proof] 
If (i) holds then $F^\pm$ is holomorphic by Lemma \ref{lem:conformal}. Conversely, if $f$ admits 
a holomorphic lift $\tilde f$, then $\tilde f=F^\pm$ by Lemma \ref{lem:P1}. 
Friderich showed \cite[Proposition 3]{Friedrich1984} that if $F^-$ is holomorphic then the vertical 
derivative $(dF^-_p)^v$ equals the mean curvature vector of $f$ at $p\in M$. Since a surface is minimal if 
and only if its mean curvature vector vanishes, the equivalence (i)$\Leftrightarrow$(ii) follows for the $-$ sign. 
It also holds for the $+$ sign since $Z^+(X)=Z^-(\overline X)$ (cf.\ Proposition \ref{prop:summary} (a)) 
and the space of minimal surfaces in $X$ does not depend on the choice of orientation of $X$. 
\end{proof}

%
%
\begin{remark}\label{rem:umbilic}
A conformal immersion $M\to X$ whose twistor lifts $M\to Z^\pm(X)$ are both holomorphic 
parameterizes a totally umbilic surface in $X$ (cf.\ \cite[Proposition 6.1]{EellsSalamon1985}).
Note also that both twistor lifts $F^\pm$ are horizontal precisely when all circles 
$I_p(v)\subset T_p M$ \eqref{eq:Ipv} are points, so the normal curvature vanishes and the surface
is totally geodesic. 
\qed \end{remark}

%
%
\begin{remark}\label{rem:horizontal}
A conformal immersion $f:M\to X$ may admit several horizontal lifts $M\to Z^\pm$, 
or no such lift. For example, if $X=\R^4$ with the flat metric 
then the horizontal distribution on $Z^\pm\cong\R^4\times \CP^1$ 
is involutive and each leaf projects diffeomorphically onto $X$ (see Example \ref{ex:R4}), so $f$ admits
a horizontal lift to every leaf; however, only the twistor lift can be holomorphic in view of Lemma \ref{lem:P1}. 
The situation is quite different if the horizontal distribution $\xi^\pm=T^h Z^\pm$ 
is a holomorphic contact bundle on $Z^\pm$.
In such case, any horizontal lift $M\to Z^\pm$ is a conformal Legendrian surface 
(tangential to the contact bundle $\xi^\pm$), hence holomorphic or antiholomorphic by 
\cite[Lemma 5.1]{AlarconForstneric2019IMRN}. By Lemma \ref{lem:P1}, this 
is the twistor lift or its antiholomorphic reflection (see Proposition \ref{prop:summary} (b)).
\qed \end{remark}

%
%
\begin{remark}\label{rem:biquadratic}
Another characterisation of superminimal surfaces is given by the vanishing of a certain quartic form 
which was first studied by Calabi \cite{Calabi1967} and Chern \cite{Chern1970A,Chern1970B}; see also 
Bryant \cite{Bryant1982JDG} and Gauduchon \cite[Proposition 7]{Gauduchon1987A}. This shows that every 
minimal immersion $S^2\to S^4$ is superminimal; 
see \cite[Theorem C]{Bryant1982JDG} or \cite[Proposition 25]{Gauduchon1987A}.  The same holds for 
minimal immersions $S^2\to \CP^2$ (see \cite[Proposition 28]{Gauduchon1987A}). 
\qed\end{remark}

We now recall two classical integrability theorems pertaining to twistor spaces.
The first one is due to Atiyah, Hitchin, and Singer  \cite[Theorem 4.1]{AtiyahHitchinSinger1978}.

%
%
\begin{theorem}\label{th:Atiyah}
The twistor space $(Z^\pm,J^\pm)$ of a smooth oriented Riemannian four-manifold $(X,g)$ is 
an integrable complex manifold if and only if the conformally invariant Weil tensor 
$W=W^++W^-$ of $(X,g)$  satisfies $W^+=0$ or $W^-=0$, respectively. 
\end{theorem}

Let us say that $(X,g)$ is $\pm$ self-dual if $W^\pm =0$.
The next result is due to Salamon \cite[Theorem 10.1]{Salamon1983}; see also  
Eells and Salamon \cite[Theorem 4.2]{EellsSalamon1985}. 

%
%
\begin{theorem}\label{th:Eells}
Assume that $(X,g)$ is a $\pm$ self-dual Riemannian four-manifold, so $(Z^\pm,J^\pm)$ is a complex manifold. 
Then, the horizontal bundle $\xi^\pm$ is a holomorphic hyperplane subbundle of $TZ^\pm$ if 
and only if $X$ is an Einstein manifold. Assuming that this holds, $\xi^\pm$ is a holomorphic contact bundle 
if and only if the scalar curvature of $X$ (the trace of the Ricci curvature) is nonzero. 
\end{theorem}

In short, the complex structures $J^\pm$ on twistor spaces $Z^\pm $ depend only on the conformal class 
of the metric on $X$, but the horizontal distribution is defined by a choice of metric in that conformal class, 
and it is holomorphic precisely when the metric is Einstein.

%
%
\begin{example}[Twistor spaces of a K\"ahler manifold]\label{ex:hermitianmanifold}
A smooth section $\sigma:X\to Z^\pm(X)$ of the twistor bundle determines an
almost hermitian structure $J_\sigma$ on $TX$ given at a point $x\in X$ by $\sigma(x)\in \Jscr^\pm(T_x X)$.
Conversely, an almost hermitian structure $J$ on $TX$ determines a section $\sigma_J:X\to Z^\pm(X)$, 
where the sign depends on whether $J$ agrees or disagrees with the orientation of $X$.
These structures are not integrable in general. 

Suppose now that $(X,g,J)$ is an integrable hermitian manifold
endowed with the natural orientation determined by $J$. Then, the associated holomorphic section 
$\sigma_J:X\to Z^+(X)$ is horizontal if and only if $(X,g,J)$ is a K\"ahler manifold. Indeed, the 
K\"ahler condition is equivalent to $J$ being invariant under the parallel transport along curves in $X$, 
which means that $\nabla J=0$. This shows that the horizontal bundle $\xi^+\subset TZ^+(X)$ 
associated to a K\"ahler manifold $X$ is never a holomorphic contact bundle. 
(Note also that $Z^+$ is in general not an integrable complex manifold.)
Any holomorphic or antiholomorphic curve in $X$ is a superminimal surface of positive spin since  
$\sigma_J$ provides a horizontal lift to $Z^+$. Another type of superminimal surfaces
of positive spin are the Lagrangian ones, i.e., those for which the image of the tangent space 
at any point by the complex structure $J$ is orthogonal to itself. If the holomorphic sectional curvature 
of $X$ is nonvanishing then any superminimal surface of positive spin in $X$ is of one of these three types 
(see \cite{EellsSalamon1985}).

On the K\"ahler manifold $\R^4=\C^2$ with the flat metric, $\xi^+$ is involutive 
(cf.\ Example \ref{ex:R4}). The twistor space $Z^+(\CP^2)$ of the projective plane 
is not integrable, and the superminimal surfaces in $\CP^2$ of positive spin are described above.
On the other hand, $Z^-(\CP^2)$ is integrable and can be identified
with the projectivised tangent bundle of $\CP^2$. There is a natural correspondence between 
superminimal surfaces of negative spin in $\CP^2$ and holomorphic curves in 
$\CP^2$ (see Gauduchon \cite[p.\ 178]{Gauduchon1987A}).
Superminimal surfaces in $\CP^2$ (and in $S^4$) were also studied by 
Montiel and Urbano \cite{MontielUrbano1997} and others.
\qed \end{example}

%
%
%
%
\section{Proofs of Theorems \ref{th:main} and \ref{th:main2}}\label{sec:proof}

%
%
\begin{proof}[Proof of Theorem \ref{th:main}]
Let $(X,g)$ be a Riemannian manifold satisfying the hypotheses of Theorem \ref{th:main}.
Let $W=W^++W^-$ denote the Weyl tensor of $X$ (see \cite[p.\ 427]{AtiyahHitchinSinger1978}). 
Assume without loss of generality that $X$ is self-dual, meaning that $W^-=0$; the analogous argument applies 
if $W^+=0$ by reversing the orientation on $X$ (see Proposition \ref{prop:summary} (a)). 
Denote by $\pi:Z=Z^-(X)\to X$ the negative twistor space of $X$ and by $\xi\subset TZ$ its horizontal bundle
(see Sect.\ \ref{sec:twistor}). Also, let $\tilde g$ denote a metric on $Z$
for which the differential $d\pi:TZ \to TX$ maps $\xi$ isometrically onto $TX$.
Such $\tilde g$ is obtained by adding to the horizontal component $\pi^* g$ a positive 
multiple $\lambda>0$ of the spherical metric on $\CP^1$. 
By Theorems \ref{th:Atiyah} and \ref{th:Eells}, $Z$ is an integrable complex manifold 
and the horizontal bundle $\xi$ is a holomorphic subbundle of the tangent bundle $TZ$.

Let $M$ be a relatively compact domain with smooth boundary 
in an ambient Riemann surface $R$, and let $f_0:\overline M\to X$ be 
a conformal superminimal immersion of negative spin, $f_0\in \SM^-(\overline M,X)$,
and of class $\Cscr^r$ for some $r\ge 3$ (see Definition \ref{def:SM}). 
Let $F_0: \overline M\to Z$ denote the twistor lift of $f_0$
(see \eqref{eq:twistorlift}). By Remark \ref{rem:regularity} the map $F_0$ is of class 
$\Cscr^{r-1}(\overline M)$, and by Theorem \ref{th:Friedrich} its restriction to 
$M$ is a horizontal holomorphic immersion $M\to Z$.

Assume first that the (constant) scalar curvature of $X$ is nonzero, so $\xi$ is a holomorphic
contact subbundle of $TZ$. According to \cite[Theorem 1.2]{Forstneric2020Merg},
the Legendrian immersion $F_0$ can be approximated in the $\Cscr^{r-1}(\overline M)$ topology
by holomorphic Legendrian immersions $F_1:U\to Z$ from open neighbourhoods
$U$ of $\overline M$ in $R$, and we may choose $F_1$ to agree with $F_0$ at any given finite set of points 
$A\subset M$. (We assumed that $r\ge 3$ since \cite[Theorem 1.2]{Forstneric2020Merg} applies to Legendrian 
immersions of class $\Cscr^2(\overline M)$.)
Projecting down to $X$ yields a conformal superminimal immersion $f_1=\pi\circ F_1:U\to X$ 
of negative spin satisfying the conclusion of the following proposition which seems worthwhile recording.

%
%
\begin{proposition}[Mergelyan approximation theorem for superminimal surfaces]
\label{prop:approx}
Assume that $(X,g)$ is a self-dual ($W^+=0$ or $W^-=0$) Einstein four-manifold. If $\overline M$ is a 
compact domain with smooth boundary in a Riemann surface $R$ and $f_0:\overline M\to X$ is a conformal 
superminimal immersion in $\SM^\pm(\overline M,X)$ (see \eqref{eq:SM}) of class $\Cscr^{r}$ for some $r\ge 3$, 
then $f_0$ can be approximated in the $\Cscr^{r-1}(\overline M)$ topology by conformal superminimal immersions 
$f\in \SM^\pm(U,X)$ from open neighbourhoods $U$ of $\overline M$ in $R$.
Furthermore, $f$ may be chosen to agree with $f_0$ to any given finite order at any 
given finite set of points in $M$.
\end{proposition}

We continue with the proof of Theorem \ref{th:main}. By \cite[Theorem 1.3]{AlarconForstneric2019IMRN}
we can approximate the holomorphic Legendrian immersion $F_1:U\to Z$ found above, uniformly
on $\overline M$, by topological embeddings $F:\overline M\to Z$ whose restrictions to $M$
are complete holomorphic Legendrian embeddings. Again, we can choose $F$ to match
$F_1$ (and hence $F_0$) at any given finite set of points in $M$. The proof of the cited theorem 
uses Darboux neighbourhoods furnished  by \cite[Theorem 1.1]{AlarconForstneric2019IMRN}, thereby 
reducing the problem to the standard contact structure on $\C^3$ for which the mentioned result is given by 
\cite[Theorem 1.2]{AlarconForstnericLopez2017CM}.

Since the differential of the twistor projection $\pi:Z\to X$ maps the horizontal bundle 
$\xi\subset TZ$ isometrically onto $TX$, the projection $f:=\pi\circ F:\overline M\to X$ is a continuous
map whose restriction to $M$ is a complete superminimal immersion $M\to X$.
By the construction, $f$ approximates $f_0$ as closely as desired uniformly on $\overline M$, 
and it can be chosen to agree with $f_0$ to any given finite order at the given finite set of points in $M$. 

By using also the general position theorem for holomorphic Legendrian 
immersions (see \cite[Theorem 1.2]{AlarconForstneric2019IMRN}) and the transversality 
argument given (for the special case of the twistor map $\CP^3\to S^4$) in 
\cite[proof of Theorem 7.5]{AlarconForstnericLarusson2019X}, we can arrange 
that the boundary $f|_{bM}:bM\to X$ is a topological embedding whose image 
consists of finitely many Jordan curves. As shown in 
\cite[proof of Theorem 1.1]{AlarconForstneric2020RMI}, we can also arrange
that the Jordan curves in $f(bM)$ have Hausdorff dimension one. 

It remains to consider the case when the manifold $(X,g)$ has vanishing scalar curvature. 
By Theorem \ref{th:Eells} the horizontal distribution $\xi$ on the twistor space $Z$ is then  
an involutive holomorphic subbundle of codimension one in $TZ$, hence defining a holomorphic foliation of 
$Z$ by smooth complex surfaces. The (horizontal, holomorphic) twistor lift $F_0$ of $f_0$ lies in a leaf of this foliation. 
It is known (see \cite{AlarconForstneric2019JAMS,AlarconForstneric2020RMI}) that 
complex curves parameterized by bordered Riemann surfaces 
in any complex manifold of dimension $>1$ enjoy the Calabi-Yau property. Projecting such a surface 
(contained in the same leaf of $\xi$ as $F_0(M)$) to $X$ gives an immersed complete superminimal surface, 
and we can arrange by a general position argument (see the proof of Theorem \ref{th:main}) that its 
boundary is topologically embedded.
\end{proof}

The argument in the above proof gives the following lemma.

%
%
\begin{lemma}[Increasing the intrinsic diameter of a superminimal surface]
\label{lem:diameter}
Let $M$ and $(X,g)$ be as in Theorem \ref{th:main}. Every conformal superminimal immersion
$f_0\in \SM^\pm(\overline M,X)$ of class $\Cscr^3$ can be approximated as closely as desired 
uniformly on $\overline M$ by a smooth conformal superminimal immersion $f\in \SM^\pm(\overline M,X)$
with embedded boundary $f(bM)\subset X$ such that the intrinsic diameter of the Riemannian surface 
$(M,f^*g)$ is arbitrarily big.
\end{lemma}

By an inductive application of this lemma, we obtain the following generalisation of Theorem \ref{th:main}.
Let $R$ be a compact Riemann surface and $M=R\setminus \bigcup_{i=0}^\infty D_i$
be an open domain of the form \eqref{eq:countable} in $R$ whose complement is a countable union of pairwise disjoint, 
smoothly bounded closed discs $D_i$. For every $j\in \Z_+$ we consider the compact domain in $R$ given by 
$
	M_j=R\setminus\bigcup_{k=0}^j \mathring D_k.
$
This is a compact bordered Riemann surface with boundary $bM_j=\bigcup_{k=0}^j bD_k$, and
$
	M_0 \supset M_1\supset M_2\supset \cdots \supset \bigcap_{j=1}^\infty M_j = \overline M.
$

%
%
\begin{theorem}[Assumptions as above] \label{th:main2}
Assume that $(X,g)$ is an Einstein four-manifold with the Weyl tensor $W=W^++W^-$.
If $W^\pm=0$ then every $f_j \in \SM^\pm(M_j,X)$ $(j\in \Z_+)$ of class $\Cscr^3$ can be approximated 
as closely as desired uniformly on $\overline M$ by continuous maps $f:\overline M\to X$ such that 
$f:M\to X$ is a complete conformal superminimal immersion in $\SM^\pm(M,X)$ 
and $f(bM)=\bigcup_i f(bD_i)$ is a union of pairwise disjoint Jordan curves of Hausdorff dimension one. 
\end{theorem}

%
%
\begin{proof} 
We outline the main idea and refer for the details to \cite[proof of Theorem 5.1]{AlarconForstneric2020RMI} 
where the analogous result is proved for conformal minimal surfaces in Euclidean spaces. 

Let $f_j\in \SM^\pm(M_j,X)$ be a smooth conformal superminimal immersion. 
Using Lemma \ref{lem:diameter} we inductively construct a sequence $f_i\in \SM^\pm(M_i,X)$ 
$(i=j+1,j+2,\ldots$) such that at every step the map $f_i:M_i\to X$ approximates 
the previous map $f_{i-1}:M_{i-1}\to X$ uniformly on $M_i \subset M_{i-1}$ as closely as desired, 
the intrinsic diameter of $(M_i,f_i^* g)$ is a big as desired, and the boundary $f_i(bM_i)\subset X$ is embedded.
(Note that at each step a new disc is taken out and hence an additional boundary curve appears.) 
By choosing the approximations to be close enough at every step and the intrinsic diameters of the 
Riemannian surfaces $(M_i,f_i^* g)$ growing fast enough, the sequence $f_i$ converges uniformly on $\overline M$ 
to a limit $f=\lim_{i\to\infty}f_i: \overline M\to X$ satisfying the conclusion of the theorem. 
For the details of this argument in an analogous situation we refer to \cite[Sect.\ 3]{AlarconForstneric2020RMI}.
\end{proof}

%
%

\section{Twistor spaces of the $4$-sphere and of the hyperbolic $4$-space}\label{sec:S4}

It was shown by Penrose \cite[Sect.\ VI]{Penrose1967}, and more explicitly by 
Bryant \cite[Sect.\ 1]{Bryant1982JDG}  that the twistor space 
of the four-sphere $S^4$ with the spherical metric can be identified with the complex projective space 
$\CP^3$ with the Fubini-Study metric 
(defined by the homogeneous $(1,1)$-form $\omega =  dd^c \log|z|^2$ on $\C^4_*$) 
such that the horizontal distribution $\xi\subset T\CP^3$ 
of the twistor projection $\pi:\CP^3\to S^4$ is a holomorphic contact bundle given
in homogeneous coordinates $[z_1:z_2:z_3:z_4]$ by the homogeneous $1$-form 
\begin{equation}\label{eq:alpha}
	\alpha = z_1dz_2-z_2dz_1+z_3dz_4-z_4dz_3.
\end{equation}
(This complex contact structure $\CP^3$ is unique up to holomorphic contactomorphisms; 
see LeBrun and Salamon \cite[Corollary 2.3]{LeBrunSalamon1994IM}.)
Proofs can also be found in many other sources, see Eells and Salamon 
\cite[Sect.\ 9]{EellsSalamon1985}, Gauduchon \cite[pp.\ 170-175]{Gauduchon1987A},
Bolton and Woodward \cite{BoltonWoodward2000}, Baird and Wood \cite[Example 7.1.4]{BairdWood2003}, 
among others. 

Due to the overall importance of this example we offer here a totally elementary explanation using only  
basic facts along with Lemma \ref{lem:complexstructures}. We consider $Z^+(S^4)$; 
the same holds for $Z^-(S^4)$ by applying \eqref{eq:reversing} to the antipodal orientation 
reversing isometry on $S^4$. In Example \ref{ex:H4} we also take a look at the twistor
space of the hyperbolic four-space $H^4$.

%
%
\begin{example}[The twistor space of $S^4$]\label{eq:S4}
The geometric scheme follows Bryant \cite{Bryant1982JDG} and
Gauduchon \cite[p.\ 171--175]{Gauduchon1987A}, \cite{Gauduchon1987B}. 
We identify the quaternionic plane $\H^2$ with $\C^4$ by 
\begin{equation}\label{eq:H2C4}
	\H^2\ni q=(q_1,q_2)=(z_1+z_2\jgot,z_3+z_4\jgot) = (z_1,z_2,z_3,z_4)=z \in \C^4,
\end{equation}
and we identify $S^4$ with the unit sphere in $\R^5=\C\oplus \C\oplus\R$
oriented by the outward vector field. Write $\H^2_*=\H^2 \setminus \{0\}$ and consider the 
commutative diagram
\[ 
	\xymatrix{
	\C^4_* \ar[r]^{\cong} &  \H^2_* \ar[r]^{\phi_1} \ar[d]_{\phi} & \CP^3 \ar[ld]_{\phi_2} \ar[d]^{\pi}
	\\ 
	\!\!\!\!\!\!\!\!\!\!\!\!\!\C^2\cup\{\infty\} \ar[r]^{\cong} & \H\P^1  \ar[r]^{\psi} & S^4 
	}
\]
where
\begin{itemize}
\item $\phi_1:\H^2_*\cong\C^4_*\to \CP^3$ is the canonical projection with fibre $\C^*$
sending $q=(q_1,q_2)\in \H^2_*$ to the complex line $\C q\in \CP^3$;
\item $\phi_2:\CP^3\to \HP^1$ is the fibre bundle sending a complex line $\C q$, $q\in \H^2_*$, to the 
quaternionic line $\H q=\C q \oplus \C \jgot q$. Thus, $\HP^1$ is the quaternionic one-dimensional projective 
space which we identify with $\H\cup\{\infty\} = \R^4\cup\{\infty\}$ such that 
$\H_2:=\{0\}\times \H=\H\,\cdotp (0,1)$ corresponds to $\infty$. The fibre $\phi_2^{-1}(\phi_2(q))$ is the 
linear rational curve $\CP^1\subset \CP^3$ of complex lines in the quaternionic line $\H q$;
\item $\phi=\phi_2\circ \phi_1:\H^2_* \to \HP^1$ sends $q\in \H^2_*$ to $\H q\in \HP^1$.
Restricting $\phi$ to the unit sphere $S^7\subset \H^2_*$ gives a Hopf map $S^7\to S^4$ with fibre $S^3$;
\item $\psi:\HP^1 \cong \R^4\cup\{\infty\} \to S^4\subset \R^5$ is the orientation preserving 
stereographic projection mapping $\infty$ to the south pole $\sgot = (0,0,0,0,-1)\in S^4$;
\item $\rho:=\psi\circ \phi: \H^2_*\to S^4$.
\end{itemize}
The stereographic projection $\psi:\R^4\cup\{\infty\}\to  S^4\subset \R^5$ with $\psi(\infty)=\sgot$ is given by
\begin{equation}\label{eq:psi}
	\psi(x) = \left(\frac{2x}{1+|x|^2},\frac{1-|x|^2}{1+|x|^2}\right).
\end{equation}
Using coordinates \eqref{eq:H2C4} it is elementary to find the following explicit formulas:
\begin{eqnarray}
	\label{eq:phi}
	\phi(q_1,q_2) &=& q_1^{-1}q_2=\frac{1}{|q_1|^2}\bar q_1 q_2 \cr 
			      &=& \frac{1}{|z_1|^2+|z_2|^2}\left(\bar z_1 z_3+z_2\bar z_4,\bar z_1z_4-z_2\bar z_3\right), \\
	\label{eq:rho}
	\rho(q_1,q_2) &=& \frac{1}{|q_1|^2+|q_2|^2}\left(2\bar q_1 q_2,|q_1|^2-|q_2|^2\right)\in S^4\subset \R^5, \\
	\label{eq:pi}
	\pi([z_1:z_2:z_3:z_4]) &=& \frac{1}{|z|^2} \left(
	2(\bar z_1 z_3+z_2\bar z_4),2(\bar z_1z_4-z_2\bar z_3),|q_1|^2-|q_2|^2\right).
\end{eqnarray}

We begin by considering the fibre $\pi^{-1}(\ngot) \subset \CP^3$ over the point 
$\ngot:=(0,0,0,0,1)\in S^4\subset \R^5$. 
This fibre is the space of complex lines in $\H_1:=\H\times \{0\}$ (hence isomorphic to $\CP^1$), 
and its normal space at every point in the Fubini-Study metric is $\H_2=\{0\}\times \H$. 
Using \eqref{eq:H2C4} we have that $\H_1=\{z_3=z_4=0\}$, and 
the form $\alpha$ \eqref{eq:alpha} along $\H_1$ equals $z_1dz_2-z_2dz_1$. 
It's kernel is the complex $3$-plane $\C\,\cdotp (z_1,z_2)\oplus \H_2$, so  
$\xi=\ker\alpha\subset T\CP^3$ coincides with $\H_2$ at every point of $\pi^{-1}(\ngot)$. 
This shows that $\xi$ is orthogonal to the fibre $\pi^{-1}(\ngot)$ in the Fubini-Study metric. 
We identify the tangent space $T_\ngot S^4=\R^4\times \{0\}$ 
with $\H$ and let $J_\igot\in \Jscr^+(T_\ngot S^4)$ denote the  
almost hermitian structure $J_\igot(1)=\igot$, $J_\igot(\jgot)=\kgot$.
Fix a point $q\in \H_1$ with $|q|=1$. Consider the differential
\[
	d\rho_{(q,0)} : T_{(q,0)}\H^2=\H_1\oplus \H_2\to T_\ngot S^4 \cong\H. 
\] 
We see from \eqref{eq:rho} that the restriction of $d\rho_{(q,0)}$ to the horizontal subspace
$\H_2=\xi$ equals
\[
	\H_2\ni q_2\mapsto 2\bar q q_2,
\]
so it is an isometry with an appropriate choice of the constant for the metrics.
If $J_q$ is the almost hermitian structure on $T_\ngot(S^4)\cong\H$ furnished by 
Lemma \ref{lem:complexstructures}, then 
\[
	d\rho_{(q,0)} \circ J_\igot = J_q\circ d\rho_{(q,0)} \quad \text{on $\H_2$}. 
\] 
This implies the restriction of $d\pi_{(q,0)}$ to the horizontal subspace $\H_2=\xi$
intertwines $J_\igot$ with $J_q$ as in the definition of the twistor space (see \eqref{eq:twistor}). 
Hence, $\pi:\CP^3\to S^4$ satisfies all properties of the twistor bundle 
$Z^+(S^4)\to S^4$ along the fibre $\pi^{-1}(\ngot)$. 

To complete the proof, it suffices to show that the situation is the same on every fibre of 
the projection $\pi:\CP^3\to S^4$. To this end, we must find a group of $\C$-linear isometries
of $\C^4\cong \H^2$, hence a subgroup of $U(4)$, which 
commutes with the left multiplication of $\H$ on $\H^2$ and
passes down to a transitive group of isometries of $S^4$.
This requirement is fulfilled by the subgroup of $U(4)$ 
preserving the quaternionic inner product on $\H^2$ given by 
\[
	\H^2\times \H^2\ni (p,q)\ \longmapsto\ p \bar q^t = p_1\bar q_1+p_2\bar q_2 \in\H.
\]
(We consider elements of $\H^2$ as row vectors acted upon by right multiplication.) 
Writing
\[
	p=(z_1+z_2\jgot,z_3+z_4\jgot)=z,\quad  q=(w_1+w_2\jgot,w_3+w_4\jgot)=w, 
\]
a calculation gives
\begin{equation}\label{eq:pbarq}
	p \bar q^t = z \, \overline w^t + \alpha_0(z,w) \jgot,\quad\ \alpha_0(z,w)=z_2w_1-z_1w_2+z_4w_3-z_3w_4. 
\end{equation}
Note that $\alpha_0(z,dz)=\alpha$ is the contact form \eqref{eq:alpha}. If $J_0\in SU(4)$ denotes the matrix
with $\left(\begin{smallmatrix}0&-1\\1&0\end{smallmatrix}\right)$ as the diagonal blocks and zero off-diagonal
blocks, then $\alpha_0(z,w)=z J_0 w^t$. It follows that the group we are looking for is
\[
	G = \{A\in U(4): AJ_0A^t=J_0\} = U(4) \cap Sp_2(\C),
\]
where $Sp_2(\C)$ is the complexified symplectic group. 
Its projectivization $\P\, G$ acts on $\CP^3$ by holomorphic contact isometries. 
This shows that $\CP^3$ is indeed the twistor space of $S^4$. 

Explicit formulas for the twistor lift of an immersions $M\to S^4$ into $\CP^3$ can be
found in \cite[Sect.\ 2]{Bryant1982JDG}, \cite[Sect.\ 9]{EellsSalamon1985}, 
\cite[Proposition 2.1]{BoltonWoodward2000}, 
among others. The antiholomorphic fibre preserving involution 
$\iota:\CP^3\to \CP^3$ (cf.\ Proposition \ref{prop:summary} (b)) is given by
\[
	\iota([z_1:z_2:z_3:z_4]) =  [-\bar z_2:\bar z_1:-\bar z_4:\bar z_3].
\]
The formula \eqref{eq:pi} immediately shows that $\pi\circ \iota=\Id_{S^4}$.
Identifying $S^4$ with $\R^4\cup \{\infty\}=\C^2\cup \{\infty\}$
via the stereographic projection $\psi$ \eqref{eq:psi} and using complex coordinates $w=(w_1,w_2)\in\C^2$, 
the spherical metric of constant sectional curvature $+1$ is given by 
\[
	g_s = \frac{4|dw|^2}{\left(1+|w|^2\right)^2},\quad\ w\in\C^2, 
\]
and \eqref{eq:phi} shows that the twistor projection 
$\phi_2=\psi^{-1}\circ\pi:\CP^3\to \C^2\cup \{\infty\}$ is given
in homogeneous coordinates $[z_1:z_2:z_3:z_4]$ on $\CP^3$ by
\begin{equation}\label{eq:tildepi}
	w_1=\frac{\bar z_1 z_3+z_2\bar z_4}{|z_1|^2+|z_2|^2},\quad\ 
	w_2=\frac{\bar z_1z_4-z_2\bar z_3}{|z_1|^2+|z_2|^2},	
	\quad\ |w|^2=\frac{|z_3|^2+|z_4|^2}{|z_1|^2+|z_2|^2}.
\end{equation}
\end{example}

%
%
\begin{example}[The twistor space of $H^4$]\label{ex:H4}
The geometric model of the {\em hyperbolic space} $H^4$ of constant sectional curvature $-1$ is the hyperquadric
\begin{equation}\label{eq:H4}
	H^4=\{x=(x_1,\ldots,x_5)\in \R^{5}: x_1^2+x_2^2+x_3^2+x_4^2+1=x_5^2,\ \ x_5>0\}
\end{equation}
in the Lorentzian space $\R^{4,1}$, that is, $\R^5$ endowed with the Lorenzian inner product 
\[
	x\circ y = x_1y_1+\cdots+x_4y_4-x_5y_5.
\]
(See Ratcliffe \cite[Sect.\ 4.5]{Ratcliffe2006}.) Note that $H^4$ is one of the two connected 
components of the unit sphere $\{x\in \R^{4,1}:x\circ x=-1\}$ of imaginary radius $\igot=\sqrt{-1}$,
the other component being given by the same equation \eqref{eq:H4} with $x_5<0$.

Consider the stereographic projection 
$\tilde \psi:\B=\{x\in \R^4:|x|^2<1\} \stackrel{\cong}{\to} H^4$ given by 
\begin{equation}\label{eq:psih}
	\tilde\psi(x) = \left(\frac{2x_1}{1-|x|^2},\cdots, \frac{2x_4}{1-|x|^2},\frac{1+|x|^2}{1-|x|^2}\right),
		\quad x\in \B.
\end{equation}
The pullback by $\tilde \psi$ of the Lorentzian pseudometric $\|x\|^2 = x\circ x$ on $\R^{4,1}$ is 
the hyperbolic metric of constant curvature $-1$ on the ball $\B$:
\[ 
	g_h =  \frac{4|dx|^2}{\left(1-|x|^2\right)^2},\quad\ x\in \B.
\] 
The Riemannian manifold $(\B,g_h)$ is the {\em Poincar\'e ball model} for $H^4$. We see from 
\eqref{eq:tildepi} that the preimage of $\B$ by the projection $\phi_2:\CP^3\to \C^2\cup\{\infty\}$ is the domain 
\begin{equation}\label{eq:Omega}
	\Omega=\phi_2^{-1}(\B) = \left\{[z_1:z_2:z_3:z_4]\in \CP^3: |z_1|^2+|z_2|^2 > |z_3|^2+|z_4|^2\right\}.
\end{equation}
Since the hyperbolic metric is conformally flat, $\Omega$ is the twistor space $Z^+(H^4)$ as a complex 
manifold (cf.\ Theorem \ref{th:Atiyah}). The twistor metric $\tilde g$ on $\Omega$ is obtained from the 
hyperbolic metric $g_h$ on the base $\B$ and the Fubini-Study metric on the fibres $\CP^1$. Explicit 
formulas for the metric $\tilde g$ and the horizontal bundle $\tilde \xi \subset T\Omega$ 
can be found in \cite[Sect.\ 4]{Friedrich1997}. 
(In the cited paper, the opposite inequality is used in \eqref{eq:Omega} which amounts to interchanging the 
variables $q_1,q_2$ in \eqref{eq:phi}, i.e., passing to another affine coordinate chart of $\HP^1$.)
The metric $\tilde g$ on $\Omega$ is a complete K\"ahler metric, and $\tilde \xi$ is a holomorphic contact bundle.

\begin{corollary}\label{cor:H4}
Superminimal surfaces of both positive and negative spin in the hyperbolic $4$-space $H^4$ 
satisfy the Calabi-Yau property. Furthermore, the twistor contact manifold $(\Omega,\tilde\xi)$ of $H^4$ is 
Kobayashi hyperbolic. The same holds for domains in any complete Riemannian four-manifold of 
constant negative sectional curvature (a space-form).
\end{corollary}

For the notion of Kobayashi hyperbolicity of complex contact manifolds, see \cite{Forstneric2017JGEA}.

\begin{proof}
The first statement follows directly from Theorems \ref{th:main} and \ref{th:main2}. 
Let $M$ be a Riemann surfaces and $f:M\to (H^4,g_h)$ be a conformal minimal immersion. 
The induced metric $f^*g_h$ on $M$ is then a K\"ahler metric with curvature bounded above by $-1$,  
the curvature of $H^4$ (see \cite[Corollary 2.2]{ChenBangYen2017}). 
By the Ahlfors lemma (see \cite[Theorem 2.1, p.\ 3]{Kobayashi2005}) it follows that any holomorphic map
$h:\D=\{z\in \C:|z|<1\} \to M$ from the disc satisfies an upper bound on the derivative at any point $p\in \D$ 
depending only on $h(p)\in M$. Hence, $M$ is Kobayashi hyperbolic and its universal
covering is the disc. Since superminimal surfaces in $H^4$ lift isometrically to holomorphic Legendrian curves 
in $(\Omega,\tilde\xi)$, the contact structure $\tilde\xi$ is hyperbolic. 
(Note that $\Omega$ itself is not Kobayashi hyperbolic since the fibres of $\phi_2:\Omega\to\B$ 
are rational curves.) The same argument applies to domains in any space-form $X$ since its universal 
metric covering space is $H^4$; see \cite[Theorem 4.1]{doCarmo1992}.
\end{proof}
\end{example}

%
%
\subsection*{Acknowledgements}
This research is supported by research program P1-0291 and grant J1-9104
from ARRS, Republic of Slovenia, and by the Stefan Bergman Prize 2019 from the 
American Mathematical Society. I wish to thank Antonio Alarc\'on, Joel Fine, Finnur L\'arusson,
and Francisco Urbano for helpful discussions on this subject.




\vspace*{5mm}
\noindent Franc Forstneri\v c

\noindent Faculty of Mathematics and Physics, University of Ljubljana, Jadranska 19, SI--1000 Ljubljana, Slovenia, and 

\noindent 
Institute of Mathematics, Physics and Mechanics, Jadranska 19, SI--1000 Ljubljana, Slovenia

\noindent e-mail: {\tt franc.forstneric@fmf.uni-lj.si}

\end{document}